\documentclass[11pt,a4paper]{article}

\usepackage{graphicx,latexsym,euscript,makeidx,color,bm}
\usepackage{amsmath,amsfonts,amssymb,amsthm,thmtools,mathtools,mathrsfs,enumerate}
\usepackage[colorlinks,linkcolor=blue,citecolor=red]{hyperref}

\usepackage{geometry}
\allowdisplaybreaks[4]
\def\dbC{\mathbb{C}}   \def\cC{{\cal C}}  
     
\def\dbE{\mathbb{E}}     
\def\dbF{\mathbb{F}} \def\sF{\mathscr{F}}

   \def\cI{{\cal I}}  
     
   \def\cK{{\cal K}}  
   \def\cL{{\cal L}}  
     
     \def\bu{\bar u}
     
\def\dbP{\mathbb{P}}     \def\bz{\bar z} 
     
\def\dbR{\mathbb{R}}   
\def\dbS{\mathbb{S}}

   \def\cW{{\cal W}}

\def\ss{\smallskip}   \def\lt{\left}                           \def\diag{{\rm diag}}
\def\ms{\medskip}     \def\rt{\right}                     
          \def\lan{\langle}       \def\as{\text{a.s.}}
\def\q{\quad}         \def\ran{\rangle}       \def\tr{{\rm tr}}
\def\qq{\qquad}             \def\les{\leqslant}
\def\no{\noindent}          \def\ges{\geqslant}
\def\hp{\hphantom}         
\def\nn{\nonumber}         \def\scp{\scriptscriptstyle}
\def\rf{\eqref}            
\def\cd{\cdot}             
\def\deq{\triangleq}  \def\({\Big(}           
       \def\){\Big)}           \def\im{{\rm im\,}}
\def\wt{\widetilde}   \def\[{\Big[}           \def\bp{\begin{pmatrix}}
\def\Ra{\Rightarrow}  \def\]{\Big]}           \def\ep{\end{pmatrix}}
\def\wh{\widehat}              \def\rank{{\rm rank\,}}
           \def\Span{{\rm span\,}}

\def\scS{\scp S} 
 
\def\a{\alpha}       \def\l{\lambda}    \def\D{\varDelta}
\def\b{\beta}               \def\F{\varPhi}
\def\d{\delta}           \def\G{\varGamma}
\def\e{\varepsilon}      \def\L{\varLambda}
             \def\Om{\varOmega}
       \def\si{\sigma}    \def\Si{\varSigma}
\def\i{\infty}             \def\Th{\varTheta}
                          
\newtheoremstyle{thry}
{}      
{}      
{\sl}   
{}      
{\bf}   
{.}     
{.5em}  
{}      
\theoremstyle{thry}

\newtheorem{theorem}{Theorem}[section]
\newtheorem{proposition}[theorem]{Proposition}
\newtheorem{corollary}[theorem]{Corollary}
\newtheorem{lemma}[theorem]{Lemma}

\theoremstyle{definition}
\newtheorem{definition}[theorem]{Definition}
\newtheorem{example}[theorem]{Example}

\theoremstyle{remark}
\newtheorem{remark}[theorem]{Remark}

\def\punct{}
\newtheoremstyle{dotless}{}{}{\rm}{}{\bf}{\punct}{.5em}{}
\theoremstyle{dotless}

\makeatletter
   
   \@addtoreset{equation}{section}
   \newcommand{\setword}[2]{%
   \phantomsection
   #1\def\@currentlabel{\unexpanded{#1}}\label{#2}%
   }
\makeatother



\begin{document}  

\title{\bf Hautus-Type Criteria for Controllability \\ 
and Stabilizability of Backward-Structured \\
Stochastic Systems}

\author{Jingrui Sun\thanks{
Department of Mathematics and SUSTech International Center for Mathematics,
Southern University of Science and Technology, 
Shenzhen, Guangdong, 518055, China (Email: sunjr@sustech.edu.cn).
This author is supported by NSFC grants 12322118 and 12271242, 
and by Shenzhen Science and Technology Program grant JCYJ20250604144337051.}   
}

\maketitle  

\no{\bf Abstract.} 
This paper develops sharp Hautus-type criteria, stochastic counterparts of the
classical Popov--Belevitch--Hautus test, for exact controllability and stabilizability
of backward-structured stochastic linear systems. The main finding is that the
stochastic Hautus obstruction is not a left eigenvector, as in deterministic linear
systems, nor an arbitrary symmetric eigenmatrix, but a positive semidefinite
eigenmatrix of a Lyapunov-type operator. We prove that exact controllability is
equivalent to the absence of such nonzero positive semidefinite eigenmatrices that
are orthogonal to the control directions. This cone restriction is sharp: excluding
all symmetric eigenmatrices with the same orthogonality property is sufficient but
not necessary. We further show that stabilizability is characterized by the same
cone-restricted Hautus condition imposed only on the nonstable spectral part of the
Lyapunov-type operator. Thus the stochastic Hautus theory developed here is governed
by a simultaneous spectral restriction and cone restriction. In addition to these
criteria, we provide finite-rank and Gramian characterizations underlying exact
controllability, establish the corresponding controllability decomposition, and show
that exact controllability implies stabilizability.

\ms
\no{\bf Key words.}
Backward-structured SDEs, exact controllability, stabilizability, 
Hautus-type criteria, controllability decomposition.

\ms
\no{\bf MSC codes.}  93E03, 93B05, 93D15, 60H10.

\section{Introduction}\label{Sec:Intro} 

Let $(\Om,\sF,\dbP)$ be a complete probability space on which a standard one-dimensional 
Brownian motion $W(\cd)=\{W(t);\,t\ges 0\}$ is defined, and let $\dbF\equiv\{\sF_t\}_{t\ges0}$ 
denote its natural filtration augmented by all $\dbP$-null sets in $\sF$. 
Denote by $L^2_{\sF_t}(\Om;\dbR^n)$ the space of all $\sF_t$-measurable, square-integrable 
$\dbR^n$-valued random variables, and by $L^2_\dbF(0,T;\dbR^m)$ the space of all $\dbF$-progressively 
measurable, square-integrable $\dbR^m$-valued processes on $[0,T]$.
Consider the controlled linear stochastic differential equation (SDE, for short) 
\begin{equation}\label{SDE:ABCD}\left\{\begin{aligned}
dX(t) &= [AX(t)+Bu(t)]dt + [CX(t)+Du(t)]dW(t),  \q t\ges 0, \\
 X(0) &= x,
\end{aligned}\right.\end{equation}
where the coefficients $ A,C\in\dbR^{n\times n}$ and $B,D\in\dbR^{n\times m}$ are constant matrices.
The system \eqref{SDE:ABCD} is said to be {\it exactly controllable} on the interval $[0,T]$ if, 
for every pair $(x,\xi)\in\dbR^n\times L_{\sF_T}^2(\Om;\dbR^n)$, there exists a control process
$u(\cd)\in L^2_\dbF(0,T;\dbR^m)$ such that  
$$
X(T;x,u(\cd)) = \xi \q\as  
$$

It was shown by Peng \cite{Peng1994} that the system \eqref{SDE:ABCD} can be exactly controllable on 
$[0,T]$ only if the diffusion coefficient $D$ has full row rank. 
In this case, by applying suitable linear transformations, the system \eqref{SDE:ABCD} can be 
equivalently converted into a backward-structured SDE of the form
\begin{equation}\label{SDE:state}\left\{\begin{aligned}
dX(t) &= [AX(t)+Bu(t)+Cz(t)]dt + z(t)dW(t),  \q t\ges 0, \\
 X(0) &= x.
\end{aligned}\right.\end{equation}
Peng \cite{Peng1994} further proved that system \eqref{SDE:state} is exactly controllable if and only if 
$$
\rank(B,\;AB,\;CB,\;A^2B,\;ACB,\;CAB,\;C^2B,\;\cdots)=n.
$$
Following Peng’s work, a number of studies have examined the controllability of backward-structured SDEs. 
For instance, Liu and Peng \cite{Liu-Peng2010} extended the exact controllability results to systems with 
bounded, time-varying deterministic coefficients. 
Wang, Yang, Yong and Yu \cite{Wang-Yang-Yong-Yu2017} introduced the notion of $L^p$-exact controllability 
for linear SDEs with random coefficients, and established BSDE-based observability criteria 
together with several equivalent characterizations of both exact and null controllability. 
More recently, Yu \cite{Yu2021} developed Gramian and Kalman-type controllability conditions for linear 
mean-field SDEs with deterministic coefficients, extending classical deterministic controllability theory 
to the McKean--Vlasov setting. 
Other related controllability problems for stochastic systems have also been studied from different 
but closely connected viewpoints, including controllability operators and stochastic controllability 
for linear systems \cite{Zabczyk1981,Ehrhardt-Kliemann1982,Mahmudov-Denker2000}, Kalman-type
conditions for approximate controllability \cite{Goreac2008}, and partial controllability of SDEs
together with exact controllability of FBSDEs \cite{Wang-Yu2020}.

\ms  

In the deterministic setting, controllability and stabilizability of linear time-invariant
systems have been fundamental topics since the classical works of Kalman
\cite{Kalman1960}, Hautus \cite{Hautus1969}, and Wonham \cite{Wonham1979};
see also the monograph \cite{Trentelman-Stoorvogel-Hautus2001} for a systematic treatment. 
For the ordinary differential equation
\begin{equation}\label{Intro:ODE}
\dot X(t)=AX(t)+Bu(t),
\end{equation}
Kalman's rank condition states that \eqref{Intro:ODE} is controllable if and only if
$$
\rank(B,AB,\dots,A^{n-1}B)=n.
$$
An alternative and particularly useful characterization is the Popov--Belevitch--Hautus
(PBH) test, often also called the Hautus lemma; see \cite{Hautus1969}.
It asserts that system \eqref{Intro:ODE} is controllable if and only if
$$
\rank(\l I-A,\;B)=n,\q \forall\l\in\dbC.
$$
Equivalently, there is no nonzero left eigenvector $y$ of $A$ such that $B^\top y=0$.
The corresponding stabilizability criterion is obtained by imposing the same condition
only on the unstable spectral part, namely for those $\l\in\si(A)$ with $\Re\l\ges0$; 
see, for example, \cite{Wonham1979,Zhou-Doyle-Glover1996,Trentelman-Stoorvogel-Hautus2001}.  
These algebraic tests are among the most powerful tools in deterministic linear control
theory, since they reveal controllability and stabilizability directly from the
coefficient matrices. 

\ms

For stochastic systems, however, the situation is more subtle. 
The interaction between the drift, the diffusion, and the adaptedness requirement 
prevents a direct transfer of the deterministic PBH test. 
Related spectral, Lyapunov, and PBH-type ideas for stochastic systems have appeared 
in several closely connected directions. For mean-square and exponential stability, 
as well as Lyapunov-type criteria for stochastic differential equations, 
see \cite{Khasminskii2012,Dragan-Morozan-Stoica2013}. 
For stochastic stabilizability, exact observability, detectability, and stochastic 
PBH-type criteria, see \cite{Zhang-Chen2004,Damm2007,Zhang-Zhang-Chen2008}. 
These works provide important stochastic analogues of stability, observability, 
and PBH-type ideas, whereas the present paper focuses on Hautus-type criteria for exact 
controllability and stabilizability of the backward-structured system \rf{SDE:state}.
In particular, for this system \rf{SDE:state}, the diffusion variable $z(\cd)$ appears 
both as the martingale integrand and through the drift term $Cz(\cd)$.  
Consequently, the reachable directions are no longer generated only by successive
powers of the matrix $A$, but by all finite products formed from the two matrices
$A$ and $C$.
This feature is already reflected in Peng's rank condition, where the reachable 
subspace is generated by
$$
\im B, \q \im AB, \q \im CB, \q \im A^2B, \q \im ACB, \q \im CAB, \q \im C^2B, \q \ldots.
$$
Although this gives a stochastic analogue of Kalman's rank condition, a corresponding
Hautus-type test is not immediate. 
Indeed, the usual eigenvector argument for deterministic systems must be replaced by 
an eigenmatrix argument involving a Lyapunov-type operator. 
More importantly, the relevant eigenmatrices are not arbitrary symmetric eigenmatrices, 
but positive semidefinite ones. 
This cone restriction
is a distinctive feature of the stochastic Hautus test: in the deterministic case, the
PBH condition is formulated in terms of left eigenvectors of $A$, whereas in the
present stochastic setting the necessary and sufficient tests involve positive
semidefinite eigenmatrices of a Lyapunov-type operator generated by the pair $(A,C)$. 

\ms   

The purpose of this paper is to develop a Hautus-type theory for the backward-structured
stochastic system \eqref{SDE:state}, covering both exact controllability and
stabilizability. The main contributions of this paper, presented in the order in which
they appear in the paper, are summarized as follows.
\begin{enumerate}[(i)]
\item We first show that Peng's Kalman-type rank condition, although expressed through
infinitely many products generated by $A$ and $C$, has a finite and computable core.
More precisely, in \autoref{prop:V(n-1)=Vn}, we prove that the corresponding reachable
subspaces stabilize after at most $n-1$ steps. Consequently, Peng's infinite rank
condition reduces to a finite one. Together with the controllability Gramian, this gives
in \autoref{prop:controllability-kehua} several equivalent characterizations of exact
controllability and exact null-controllability. A notable consequence is that exact
controllability is independent of the time horizon, despite the adapted stochastic
nature of the terminal target.

\item This finite-dimensional reachable subspace further leads to a controllability
decomposition for backward-structured SDEs; see \autoref{thm:control-decomposition}.
After a suitable orthogonal change of coordinates, the system can be separated into
its controllable and uncontrollable components. This is the stochastic counterpart of
the classical Kalman controllability decomposition, but with an essential difference:
the invariant subspace is generated jointly by the drift matrix $A$ and the
diffusion-related matrix $C$, rather than by powers of $A$ alone.

\item Our first main Hautus-type result is the exact controllability criterion in
\autoref{thm:H-test-controllability}. To state it, let $\dbS^n$ denote the space of
real symmetric $n\times n$ matrices, and let $\bar\dbS^n_+$ denote the cone of
positive semidefinite matrices in $\dbS^n$. We introduce the Lyapunov-type operator
$\cL_{\scp(A,C)}:\dbS^n\to\dbS^n$ by
$$
\cL_{\scp(A,C)}(M)=MA+A^\top M+C^\top MC,\q M\in\dbS^n.
$$
We prove that system \eqref{SDE:state} is exactly controllable if and only if there is 
no nonzero \emph{positive semidefinite} eigenmatrix $H$ of $\cL_{\scp(-A,C)}$ satisfying
$B^\top H=0$.

\ms 

This result identifies the sharp stochastic analogue of the deterministic PBH
obstruction. In the classical theory, one tests left eigenvectors of the system matrix.
For system \eqref{SDE:state}, however, the correct obstructions are positive semidefinite
eigenmatrices of a Lyapunov-type operator. This is not just a change of language:
the cone restriction is essential. As shown in \autoref{example}, the stronger condition 
excluding all symmetric eigenmatrices $H\in\dbS^n$ with $B^\top H=0$ is sufficient but not
necessary for exact controllability. Thus the cone $\bar\dbS^n_+$ is the sharp
testing class for the stochastic Hautus criterion.

\item Our second main Hautus-type result is the stabilizability criterion in
\autoref{thm:H-test-stability}. Since the standard notions of $L^2$-, exponential
mean-square, and asymptotic mean-square stabilizability are equivalent in the present
setting, we simply speak of stabilizability. This result is more delicate than the
exact controllability criterion. Its proof requires three intermediate ingredients:
a cone-duality characterization for Lyapunov-type inequalities in
\autoref{prop:P>0+LP>0}, a complete treatment of the uncontrolled case $B=0$ in
\autoref{prop:stabilizable:B=0}, and a reduction of stabilizability to the
uncontrollable subsystem under the controllability decomposition in
\autoref{prop:stable-decomposition}.

Combining these ingredients, we prove that system \eqref{SDE:state} is stabilizable 
if and only if, for every $\l\les0$ and every $H\in\bar\dbS^n_+$,
$$
\cL_{\scp(-A,C)}(H)=\l H,\q B^\top H=0  \q\Longrightarrow\q H=0.
$$
This theorem reveals the main mechanism behind stochastic stabilizability. In
deterministic linear control, stabilizability is obtained by testing the PBH condition
only on the unstable eigenvalues of $A$. Here, the matrix $A$ itself is no longer the
right spectral object. Stabilizability is governed by the nonstable spectral part of
the Lyapunov-type operator $\cL_{\scp(-A,C)}$, but only inside the positive
semidefinite cone. Thus the criterion contains two simultaneous restrictions: a
spectral restriction and a cone restriction. This double restriction is the key new
feature of the stochastic PBH theory developed in this paper. As a byproduct,
\autoref{prop:control-stable} shows that exact controllability implies stabilizability.
\end{enumerate}

The rest of the paper is organized as follows. 
Section~2 collects the notation and preliminary results on stability and stabilizability. 
Section~3 is devoted to exact controllability of system \rf{SDE:state}. 
Section~4 establishes the Hautus-type characterization of stabilizability.

\section{Preliminaries}  

In this section, we introduce the notation used throughout the paper and recall the notions of controllability 
and stabilizability together with some related basic results.

\ms

Let $\dbR$ and $\dbC$ denote the sets of real and complex numbers, respectively, and set
$$
\dbC^- \deq \{z\in\dbC \mid \Re z<0\}.
$$
For integers $n,m\ges1$, let $\dbR^{n\times m}$ be the space of real $n\times m$ matrices, 
endowed with the Frobenius inner product
$$
\lan M,N\ran \deq \tr(M^\top N), \q M,N\in\dbR^{n\times m}
$$
and the associated norm $|\cdot|$. Here and throughout, $M^\top$ denotes the transpose of $M$, 
and $\tr(\cdot)$ the trace.
Let $\dbS^n\subseteq \dbR^{n\times n}$ denote the subspace of symmetric matrices. 
We write $\dbS^n_+$ and $\bar{\dbS}^n_+$ for the sets of positive definite and positive semidefinite 
matrices in $\dbS^n$, respectively. 
For $M,N\in\dbS^n$, the notation $M\ges N$ (resp., $M>N$) means that $M-N$ is positive semidefinite 
(resp., positive definite).
The identity matrix in $\dbR^{n\times n}$ is denoted by $I_n$, or simply by $I$ when no confusion can arise. 
For any matrix $M$, we denote by $\im M$ and $\ker M$ the image and kernel of $M$, respectively.
All vectors are taken to be column vectors unless stated otherwise.

\ms  

Recall the probability space $(\Om,\sF,\dbP)$ and the filtration $\dbF=\{\sF_t\}_{t\ges0}$. 
For a random vector $\xi$, we write $\xi\in\sF_t$ to mean that $\xi$ is $\sF_t$-measurable. 
For a stochastic process $X$, we write $X\in\dbF$ if $X$ is progressively measurable with respect to $\dbF$. 
We also introduce the following spaces:
\begin{align*}
C([0,T];\dbS^n) &\deq\big\{u:[0,T]\to\dbS^n \mid u \text{ is continuous on } [0,T]\big\}, \\
L^2_{\sF_T}(\Om;\dbR^n)&\deq\big\{\xi:\Om\to\dbR^n \mid \xi\in\sF_T, \dbE|\xi|^2<\i\big\}, \\
L^2_\dbF(0,T;\dbR^m)&\deq\lt\{u:[0,T]\times\Om\to\dbR^m \,\middle|\, u\in\dbF, \dbE\int_0^T|u(t)|^2dt<\i\rt\}, \\
L^2_\dbF(0,\i;\dbR^m)&\deq\lt\{u:[0,\i)\times\Om\to\dbR^m \,\middle|\, u\in\dbF, \dbE\int_0^\i|u(t)|^2dt<\i\rt\}. 
\end{align*}

\ss

Although our main focus is on the backward-structured SDE \rf{SDE:state}, 
it is convenient to recall the notions of (exact) controllability and stabilizability 
for the general linear stochastic control system
\begin{equation}\label{sys:ACBD}\lt\{\begin{aligned}
dX(t) &= [AX(t)+Bu(t)]dt + [CX(t)+Du(t)]dW(t), \q t\ges0, \\
 X(0) &= x,
\end{aligned}\rt.\end{equation}
which will be denoted by $[A,C;B,D]$, where
$$ 
A,C\in\dbR^{n\times n}, \q  B,D\in\dbR^{n\times m}.  
$$
When $B=D=0$, the above system reduces to the uncontrolled system
$$
\lt\{\begin{aligned}
dX(t) &= AX(t)dt + CX(t)dW(t), \q t\ges 0,\\
X(0) &= x,
\end{aligned}\rt.
$$
which is simply denoted by $[A,C]$.

\begin{definition}
The system $[A,C;B,D]$ is said to be
\begin{enumerate}[(i)]
\item \emph{exactly controllable} on $[0,T]$ if, for every $(x,\xi)\in\dbR^n\times L_{\sF_T}^2(\Om;\dbR^n)$,
      there exists $u(\cd)\in L^2_\dbF(0,T;\dbR^m)$ such that  
      $$
      X(T;x,u(\cd)) = \xi \q\as;  
      $$
\item \emph{exactly null-controllable} on $[0,T]$ if, for every $x\in\dbR^n$, there exists 
      $u(\cd)\in L^2_\dbF(0,T;\dbR^m)$ such that   
      $$
      X(T;x,u(\cd)) = 0 \q\as  
      $$
\end{enumerate}
\end{definition}

\begin{definition} 
The system $[A,C]$ is said to be 
\begin{enumerate}[(i)]
\item \emph{exponentially mean-square stable} if there exist constants $M\ges 1$ and $\l>0$ 
      such that the solution $X(\cd\,;x)$ with initial state $x$ satisfies
      $$
      \dbE |X(t;x)|^2 \les M e^{-\l t}|x|^2, \q\forall t\ges 0,~\forall x\in\dbR^n.
      $$
\item \emph{$L^2$-stable} if the solution $X(\cd\,;x)$ with initial state $x$ satisfies
      $$
      \dbE\int_0^\i |X(t;x)|^2dt<\i,\q \forall x\in\dbR^n.
      $$
\item \emph{asymptotically mean-square stable} if the solution $X(\cd\,;x)$ with initial state $x$ satisfies
      $$
      \lim_{t\to\i}\dbE |X(t;x)|^2=0,\q \forall x\in\dbR^n.
      $$
\end{enumerate}
\end{definition}

\begin{definition}\label{def:stabilizability}
The system $[A,C;B,D]$ is said to be \emph{exponentially mean-square stabilizable}
(resp.\ \emph{$L^2$-stabilizable}, \emph{asymptotically mean-square stabilizable})
if there exists a matrix $\Th\in\dbR^{m\times n}$ such that the closed-loop system
$[A+B\Th,C+D\Th]$ is exponentially mean-square stable
(resp.\ $L^2$-stable, asymptotically mean-square stable).
Any such matrix $\Th$ is called a \emph{stabilizer} of system $[A,C;B,D]$.
\end{definition}

The following lemma shows that the above stability notions for the system $[A,C]$ are equivalent,
and also provides an additional characterization in terms of a Lyapunov inequality. 
Closely related mean-square stability and Lyapunov-type criteria can be found, for instance, 
in \cite{Khasminskii2012,Dragan-Morozan-Stoica2013,Huang-Li-Yong2015}.  
Since the relevant arguments are not readily available in a unified form, we include a proof 
for the reader's convenience.

\begin{lemma}\label{lmm:equi-stable}
The following statements are equivalent:
\begin{enumerate}[\rm(i)]
\item System $[A,C]$ is asymptotically mean-square stable.
\item The solution $\varPi(\cd)$ to the matrix SDE
      \begin{equation}\label{matrix-SDE}\left\{\begin{aligned}
      d\varPi(t) &= A\varPi(t)dt+C\varPi(t)dW(t), \q t\ges0, \\
       \varPi(0) &= I 
      \end{aligned}\right.\end{equation}
      satisfies $\dbE|\varPi(t)|^2 <1$ for some $t>0$.
\item System $[A,C]$ is $L^2$-stable.
\item There exists a matrix $P\in\dbS^n_+$ such that
      $$
      PA + A^{\top}P + C^{\top}PC <0. 
      $$
\item System $[A,C]$ is exponentially mean-square stable.
\end{enumerate}
Moreover, if any of the above statements holds, then for any $\L\in\dbS^n$, the Lyapunov equation
$$
PA+A^\top P+C^\top PC+\L=0
$$
admits a unique solution $P\in\dbS^n$, which is given by
\begin{equation}\label{rep:P}
P = \dbE\int_0^{\i} \varPi(t)^{\top}\L\varPi(t) dt. 
\end{equation}
\end{lemma}

\begin{proof}
The equivalence of (iii) and (iv), together with the representation \rf{rep:P}, was established 
in \cite[Proposition 3.5]{Huang-Li-Yong2015}. The implication (v) $\Ra$ (iii) is immediate.
Thus, it remains only to prove (iii) $\Ra$ (ii) $\Ra$ (i) $\Ra$ (iii) $\Ra$ (v).

\ms 

For (iii) $\Ra$ (ii), we observe that for any integer $k\ges1$ and any $t\ges0$,
$\G_k(t)\deq\varPi(t+k)\varPi(k)^{-1}$ has the same distribution as $\varPi(t)$,
and is independent of $\varPi(k)$. Thus,
\begin{align}\label{varPi(t+k)}
\dbE|\varPi(t+k)|^2
&= \dbE\[\tr\(\varPi(k)^\top\G_k(t)^\top\G_k(t)\varPi(k)\)\]
= \tr\(\dbE\[\varPi(k)^\top\G_k(t)^\top\G_k(t)\varPi(k)\]\) \nn\\
&= \tr\(\dbE\[\varPi(k)^\top\dbE[\G_k(t)^\top\G_k(t)]\varPi(k)\]\)
= \tr\(\dbE\[\varPi(k)^\top\dbE[\varPi(t)^\top\varPi(t)]\varPi(k)\]\) \nn\\
&= \dbE\[\tr\(\varPi(k)^\top\dbE[\varPi(t)^\top\varPi(t)]\varPi(k)\)\].
\end{align}
Since $t\mapsto\dbE[\varPi(t)^\top\varPi(t)]$ is continuous and
$\dbE[\varPi(t)^\top\varPi(t)]\in\dbS^n_+$ for every $t\ges0$, 
there exist constants $\b,\a>0$ such that 
\begin{equation}\label{26-03-30:1}
\a I\les \dbE[\varPi(t)^\top\varPi(t)]\les \b I, \q\forall t\in[0,1].
\end{equation}
Then, using the $L^2$-stability of $[A,C]$, together with \rf{varPi(t+k)}, we obtain 
\begin{align*}
\i>\dbE\int_0^\i|\varPi(t)|^2dt
&= \sum^\i_{k=0}\int_{k}^{k+1}\dbE|\varPi(t)|^2dt
= \sum^\i_{k=0}\int_0^1\dbE|\varPi(t+k)|^2dt \\
&= \sum^\i_{k=0}\int_0^1\dbE\[\tr\(\varPi(k)^\top\dbE[\varPi(t)^\top\varPi(t)]\varPi(k)\)\]dt \\
&\ges \a\sum^\i_{k=0}\int_0^1\dbE\[\tr\(\varPi(k)^\top\varPi(k)\)\]dt 
=\a\sum^\i_{k=0}\dbE|\varPi(k)|^2.
\end{align*}
Therefore, $\dbE|\varPi(k)|^2\to0$ as $k\to\i$.

\ms 

For (ii) $\Ra$ (i), we may assume without loss of generality that $\dbE|\varPi(1)|^2<1$.
Taking $t=1$ in \rf{varPi(t+k)} and iterating, we obtain
\begin{align*}
\dbE|\varPi(1+k)|^2
&\les \dbE|\varPi(1)|^2\,\dbE\[\tr\(\varPi(k)^\top\varPi(k)\)\] 
=\dbE|\varPi(1)|^2\,\dbE|\varPi(k)|^2 \les \[\dbE|\varPi(1)|^2\]^{k+1}.
\end{align*}
Thus, for any $\e>0$, we can find an integer $N\ges1$ such that
$$
\dbE|\varPi(k)|^2 \les \b^{-1}\e, \q \forall k\ges N,
$$
where $\b>0$ is the constant in \rf{26-03-30:1}. Then, for any $t\ges N$, letting
$$
k\deq \lfloor t\rfloor, \q s\deq t-k\in[0,1],
$$
we have by \rf{varPi(t+k)} that 
\begin{align*}
\dbE|\varPi(t)|^2 
   = \dbE|\varPi(s+k)|^2 
   = \dbE\[\tr\(\varPi(k)^\top \dbE\[\varPi(s)^\top\varPi(s)\]\varPi(k)\)\] 
\les \b\dbE|\varPi(k)|^2 
\les \e.
\end{align*}
Hence, $\lim_{t\to\i}\dbE|\varPi(t)|^2=0$, 
and thus system $[A,C]$ is asymptotically mean-square stable.

\ms  

For (i) $\Ra$ (iii), choose an integer $k\ges1$ such that $\dbE|\varPi(k)|^2<1$. 
For any $t\in[0,k]$ and any integer $h\ges0$, by \rf{varPi(t+k)} we have
$$
\dbE|\varPi(t+hk)|^2
   = \dbE\[\tr\(\varPi(hk)^\top \dbE[\varPi(t)^\top\varPi(t)]\varPi(hk)\)\]
\les \dbE|\varPi(t)|^2\,\dbE|\varPi(hk)|^2.
$$
Moreover, taking $t=k$ in the above and iterating, we obtain
$$
\dbE|\varPi(hk)|^2 \les \[\dbE|\varPi(k)|^2\]^h.
$$
Consequently,
\begin{align*}
\dbE\int_0^\i |\varPi(t)|^2dt
&= \sum_{h=0}^\i \int_{hk}^{(h+1)k}\dbE|\varPi(t)|^2dt
 = \sum_{h=0}^\i \int_0^k \dbE|\varPi(t+hk)|^2dt \\
&\les \sum_{h=0}^\i \int_0^k \dbE|\varPi(t)|^2\,\[\dbE|\varPi(k)|^2\]^hdt 
    = \lt(\int_0^k \dbE|\varPi(t)|^2dt\rt)\sum_{h=0}^\i \[\dbE|\varPi(k)|^2\]^h <\i.
\end{align*}
Thus, system $[A,C]$ is $L^2$-stable. 

\ms  

For (iii) $\Ra$ (v), or equivalently, for (iv) $\Ra$ (v), let $P>0$ be such that
$$
PA+A^\top P+C^\top PC+I=0,
$$
and let $\d_1,\d_2>0$ be such that $\d_1 I \les P \les \d_2 I$.
Let $X(\cd\,;x)$ be the solution of system $[A,C]$ with initial state $x\in\dbR^n$.
By It\^{o}'s rule,
\begin{align*}
{d\over dt}\dbE[X(t;x)^\top PX(t;x)]
&= \dbE\[X(t;x)^\top(PA + A^\top P + C^\top PC)X(t;x)\] = -\dbE|X(t;x)|^2 \\
&\les -\delta_2^{-1}\dbE[X(t;x)^\top PX(t;x)], \q\forall t\ges0.
\end{align*}
Hence, by Gronwall's inequality,
$$
\dbE[X(t;x)^\top PX(t;x)] \les x^\top Px\, e^{-t/\d_2}, \q\forall t\ges0.
$$
Consequently,
$$
\d_1 \dbE|X(t;x)|^2
\les \dbE[X(t;x)^\top PX(t;x)]
\les x^\top Px\, e^{-t/\d_2}
\les \d_2 |x|^2 e^{-t/\d_2}, \q\forall t\ges0,\ \forall x\in\dbR^n,
$$
which shows that system $[A,C]$ is exponentially mean-square stable. 
\end{proof}

We next recall a characterization of $L^2$-stabilizability for the controlled system $[A,C;B,D]$; 
see \cite[pp.~71--73]{Sun-Yong2020book-a} and \cite[pp.~201--204]{Sun-Yong2024JDE} for the proof.

\begin{lemma}\label{lmm:ACBD}
Let $M\in\dbS^n_+$ and $N\in\dbS^m_+$. The following statements are equivalent:
\begin{enumerate}[\rm (i)]
\item System $[A,C;B,D]$ is  $L^2$-stabilizable. 
\item The algebraic Riccati equation
      \begin{equation}\label{ACBD:Ric}
        PA + A^\top P + C^\top PC + M -(PB+C^\top PD)(N+D^\top PD)^{-1}(B^\top P+D^\top PC) = 0
      \end{equation}
      admits a solution $P\in\dbS^n_+$. 
\item For every $x\in\dbR^n$, there exists a control process $u(\cd)\in L^2_\dbF(0,\i;\dbR^m)$ such that 
      $$
      \dbE\int_0^\i|X(t;x,u(\cd))|^2dt<\i. 
      $$
\end{enumerate}
In this case, the positive definite solution $P$ to \rf{ACBD:Ric} is unique, and  
\begin{equation}\label{def:Th}
\Th \deq -(N+D^\top PD)^{-1}(B^\top P+D^\top PC)
\end{equation}
is a stabilizer of system $[A,C;B,D]$.   
\end{lemma}

\section{Exact controllability} 

In this section, we study the backward-structured SDE \rf{SDE:state}, 
with primary emphasis on deriving a Hautus-type characterization; 
see \autoref{thm:H-test-controllability}. 
We also establish a corresponding controllability decomposition in 
\autoref{thm:control-decomposition}, and derive in 
\autoref{prop:controllability-kehua} several equivalent characterizations 
of exact controllability, refining some related known results. 

\ms  

Let $\F(\cd)=\{\F(t);0\les t<\i\}$ be the solution to the matrix SDE
\begin{equation}\label{eqn:Phi}\left\{\begin{aligned}
d\F(t)  &= -\F(t)Adt-\F(t)CdW(t), \q t\ges0, \\
 \F(0)  &= I,
\end{aligned}\right.\end{equation}
and define 
\begin{equation}\label{Grammian}
G(T) \deq \dbE\lt[\int_0^T\F(t)BB^\top\F(t)^\top dt\rt] \in \bar\dbS^n_+.
\end{equation} 
We call $G(T)$ the {\it controllability Gramian} of system \rf{SDE:state} over $[0,T]$.
Recall that, for any matrix $M$, $\im M$ and $\ker M$ denote its image and kernel, respectively.
The following lemma, adapted from Bi--Sun--Xiong \cite{Bi-Sun-Xiong2020}, collects two useful 
facts on controllability and the Gramian $G(T)$.

\begin{lemma}\label{lmm:controllability-Gramian}
The following statements hold:
\begin{enumerate}[\rm(i)]
\item For any initial-terminal state pair $(x,\xi)\in\dbR^n\times L_{\sF_T}^2(\Om;\dbR^n)$, 
      there exists a control
      $(u(\cd),z(\cd))\in L^2_\dbF(0,T;\dbR^m)\times L^2_\dbF(0,T;\dbR^n)$ such that
      $$
      X(T;x,u(\cd),z(\cd))=\xi \q\as
      $$
      if and only if $x-\dbE[\F(T)\xi]\in\im G(T)$.
\item The solution $P:[0,T]\to\dbS^n$ to the linear matrix ODE
      \begin{equation}\label{G(T):formular}\left\{\begin{aligned}
      & \dot P(t) - AP(t)- P(t)A^\top + CP(t)C^\top + BB^\top = 0, \q t\in[0,T],\\
      & P(T) = 0
      \end{aligned}\right.
      \end{equation}
      is given by
      \begin{equation}\label{formula:P-Gramian}
      P(t) = \dbE\int_t^T\F(t)^{-1}\F(s)BB^\top[\F(t)^{-1}\F(s)]^\top ds, \q t\in[0,T].
      \end{equation}
      In particular,
      $$
      G(T)=P(0), \q P(t)\ges0,\ \forall t\in[0,T].
      $$
\end{enumerate}
\end{lemma}

For an integer $k\ges0$, a \emph{word of length $k$ over $\{A,C\}$} is a matrix of the form
$$
M = X_1 X_2 \cdots X_k, \q X_i\in\{A,C\}, ~ i=1,\dots,k.
$$
When $k=0$, the empty product is understood as the identity matrix $I$.
Let $\cW_k(A,C)$ denote the set of all words of length at most $k$ over $\{A,C\}$, and define
\begin{align}
V_k  &\deq \Span\big\{\im(MB): M\in\cW_k(A,C)\big\},          \label{def:Vk}\\
V_\i &\deq \Span\big\{\im(MB): M\in\cW_k(A,C),~k\ges0\big\}. \label{def:Vi}
\end{align}
Peng \cite[Theorem 3.2]{Peng1994} showed that the backward-structured SDE \rf{SDE:state}
is exactly controllable on $[0,T]$ if and only if
$$
\dim V_\i = n.
$$
The next proposition shows that the sequence $\{V_k\}_{k\ges0}$ stabilizes after at most
$n-1$ steps. Hence the above condition is equivalent to
$$
\dim V_{n-1}=n.
$$

\begin{proposition}\label{prop:V(n-1)=Vn}
There exists $k\les n-1$ such that
$$
V_k = V_{k+1} = V_{k+2} = \cdots. 
$$
\end{proposition}

\begin{proof} 
The conclusion is trivial if $B=0$, so assume $B\neq0$. Then
$$
V_0\subseteq V_1\subseteq V_2\subseteq \cdots \subseteq \dbR^n 
\q\text{and}\q 
1\les d_0\les d_1\les\cdots\les n,
$$
where $d_k\deq\dim V_k$.  
We note that if $V_{k+1}=V_k$ for some $k\ges0$, then $V_{k+2}=V_{k+1}$. Indeed,
for any $N\in\cW_{k+2}(A,C)$, writing $N=SM$ with $S\in\{A,C\}$ and
$M\in\cW_{k+1}(A,C)$, we have  
$$
MBx\in V_{k+1}=V_k, \q\forall x\in\dbR^m. 
$$
Thus, there exist finitely many $\a_i\in\dbR$, $M_i\in\cW_k(A,C)$, and $u_i\in\dbR^m$ such that 
\begin{align*}
NBx=SMBx= S\(\sum_i \a_i M_iBu_i\)=\sum_i \a_i SM_iBu_i\in V_{k+1}.
\end{align*}
Hence $V_{k+2}\subseteq V_{k+1}$, and therefore $V_{k+2}=V_{k+1}$.
If $V_k\subsetneq V_{k+1}$ for all $k=0,1,\dots,n-1$, then
$$
1\les d_0<d_1<\cdots<d_n,
$$
which implies $d_n\ges n+1$, a contradiction. Therefore the result follows. 
\end{proof}

Combining the above results, we obtain the following equivalent characterizations 
of exact controllability.

\begin{proposition}\label{prop:controllability-kehua} 
The following statements are equivalent:
\begin{enumerate}[\rm(i)]
\item System \rf{SDE:state} is exactly controllable on $[0,T]$.
\item System \rf{SDE:state} is exactly null-controllable on $[0,T]$.
\item The controllability Gramian $G(T)$ is positive definite. 
\item $\dim V_{n-1}=n$. 
\end{enumerate}
\end{proposition}

\begin{proof} 
The implication (i) $\Ra$ (ii) is obvious.  
The implications (ii) $\Ra$ (iii) and (iii) $\Ra$ (i) follow directly 
from part (i) of \autoref{lmm:controllability-Gramian}.
The equivalence between (iii) and (iv) follows directly from  
\cite[Theorem 3.2]{Peng1994} and \autoref{prop:V(n-1)=Vn}.      
\end{proof}

\begin{remark}
Condition (iv) in \autoref{prop:controllability-kehua} shows that the exact
controllability of system \rf{SDE:state} is independent of the time horizon $T>0$.
Accordingly, in what follows, we simply say that system \rf{SDE:state} is exactly
controllable, without explicitly indicating the interval $[0,T]$. 
\end{remark}

As a direct consequence of \autoref{prop:controllability-kehua}, exact controllability 
is invariant under the following feedback perturbations.

\begin{corollary}
If system \rf{SDE:state} is exactly controllable, 
then for any $F_1,F_2\in\dbR^{m\times n}$, the system
\begin{equation*}\left\{\begin{aligned}
dX(t) &= [(A+BF_1)X(t)+Bu(t)+(C+BF_2)z(t)]dt + z(t)dW(t), \\
 X(0) &= x
\end{aligned}\right.\end{equation*}
is also exactly controllable.  
\end{corollary}

\begin{proof}
By \autoref{prop:controllability-kehua}, it suffices to show that
\begin{equation}\label{A+BF:xMB=0}
x^\top MB=0, \q\forall M\in\cW_{n-1}(A+BF_1,C+BF_2)
\end{equation}
implies $x=0$. So let $x\in\dbR^n$ satisfy \rf{A+BF:xMB=0}. We claim that
\begin{equation}\label{xMB=0:AC}
x^\top MB=0,   \q\forall M\in\cW_{n-1}(A,C). 
\end{equation}
Indeed, by taking $M=I$ in \rf{A+BF:xMB=0} we obtain $x^\top B =0$, and 
\begin{align*}
0&=x^\top (A+BF_1)B = x^\top AB + x^\top BF_1B = x^\top AB, \\
0&=x^\top (C+BF_2)B = x^\top CB + x^\top BF_2B = x^\top CB.
\end{align*}
Thus, $x^\top MB =0$ for all $M\in\cW_1(A,C)$. Fix $k\ges1$ and assume
\begin{equation}\label{xMB=0:k}
x^\top MB =0, \q\forall M\in\cW_k(A,C).
\end{equation}
We next show that 
\begin{equation}\label{xMB=0:k+1}
x^\top MB =0, \q\forall M\in\cW_{k+1}(A,C).
\end{equation}
For $X_1,\dots,X_k,X_{k+1}\in\{A,C\}$, let
$$
\Th_j \deq \begin{cases} F_1, &\text{if } X_j=A,\\F_2, &\text{if } X_j=C,\end{cases} \q j=1,2\ldots,k+1.
$$
Using \rf{A+BF:xMB=0} and \rf{xMB=0:k}, we have by expansion:
\begin{align*}
0&=x^\top(X_{1}+B\Th_{1})(X_{2}+B\Th_{2})\cdots(X_{k+1}+B\Th_{k+1})B  \\
&=x^\top X_{1}(X_{2}+B\Th_{2})\cdots(X_{k+1}+B\Th_{k+1})B  
 +x^\top B\Th_{1}(X_{2}+B\Th_{2})\cdots(X_{k+1}+B\Th_{k+1})B \\
&=x^\top X_{1}X_{2}(X_{3}+B\Th_{3})\cdots(X_{k+1}+B\Th_{k+1})B  \\
&= \cdots ~ \cdots ~  \cdots \\
&=x^\top X_{1}X_{2}\cdots X_{k+1}B.
\end{align*} 
This proves \rf{xMB=0:k+1}, and hence \rf{xMB=0:AC} follows by induction.
Since system \rf{SDE:state} is exactly controllable, \autoref{prop:controllability-kehua}
implies that $\dim V_{n-1}=n$. Therefore, \rf{xMB=0:AC} yields $x=0$. 
\end{proof}

Next, we turn to a controllability decomposition for system \rf{SDE:state}.
We begin with several preliminary results.
For each initial state $x\in\dbR^n$, let  $X(\cd\,;x)$ denote the solution of 
system $[A,C]$ corresponding to $x$.

\begin{definition}
A subspace $V\subseteq\dbR^n$ is said to be \emph{$[A,C]$-invariant} if for every $x\in V$,
$$
X(t;x) \in V \q\as, \q\forall t\ges 0. 
$$
\end{definition}

\begin{lemma}\label{lmm:AC-invariant}
A subspace $V\subseteq\dbR^n$ is $[A,C]$-invariant if and only if 
$V$ is invariant under both $A$ and $C$; that is,
$$
AV\subseteq V, \q CV\subseteq V.
$$
\end{lemma}

\begin{proof}
We may assume that $0<\ell\deq\dim V<n$. 
Choose an orthogonal matrix $P\in\dbR^{n\times n}$ such that $V=P\bp\dbR^\ell \\[-2mm] 0\ep$.
Write
$$
P^{-1}AP=\bp A_{11} & A_{12} \\ A_{21} & A_{22}\ep, \q
P^{-1}CP=\bp C_{11} & C_{12} \\ C_{21} & C_{22}\ep.
$$
where $A_{11},C_{11}\in\dbR^{\ell\times\ell}$, and the remaining blocks 
have compatible dimensions. Then 
$$
AV\subseteq V \iff A_{21}=0, \q CV\subseteq V \iff C_{21}=0.
$$
For any $x\in V$, let $X(\cd)$ be the solution of system $[A,C]$ with initial state $x$, and write
$$
P^{-1}x=\bp y \\ 0\ep, \q P^{-1}X(t)=\bp Y_1(t) \\ Y_2(t)\ep,
$$
where $y\in\dbR^\ell$ and $Y_1(t)\in\dbR^\ell$ for all $t\ges0$.  Then
\begin{align}
dY_1(t) &= [A_{11}Y_1(t)+A_{12}Y_2(t)]dt+[C_{11}Y_1(t)+C_{12}Y_2(t)]dW(t), \label{dY1=}\\
dY_2(t) &= [A_{21}Y_1(t)+A_{22}Y_2(t)]dt+[C_{21}Y_1(t)+C_{22}Y_2(t)]dW(t). \label{dY2=}
\end{align}

If $V$ is $[A,C]$-invariant, then $Y_2(\cd)\equiv0$, and hence \rf{dY2=} yields
$$
A_{21}Y_1(t)=0, \q C_{21}Y_1(t)=0, \q\as,~\forall t\ges0.
$$
In particular, at $t=0$,
$$
A_{21}y=0, \q C_{21}y=0.
$$
Since $y\in\dbR^\ell$ is arbitrary, it follows that $A_{21}=C_{21}=0$. 
Thus, $AV\subseteq V$ and $CV\subseteq V$.

\ms  

Conversely, if $AV\subseteq V$ and $CV\subseteq V$, then $A_{21}=C_{21}=0$, so \rf{dY2=} reduces to
$$
dY_2(t)=A_{22}Y_2(t)dt+C_{22}Y_2(t)dW(t), \q Y_2(0)=0.
$$
Hence $Y_2(\cd)\equiv0$, and therefore
$$
X(t)=P\bp Y_1(t)\\0\ep\in V, \q\forall t\ges0.
$$
Thus $V$ is $[A,C]$-invariant.
\end{proof}

\begin{lemma}\label{lmm:Vn-invariant}
Let $V_k$ be the space defined by \rf{def:Vk}. Then $V_{n-1}$ is $[A,C]$-invariant. 
\end{lemma}

\begin{proof}
For any $x\in V_{n-1}$, we may write
$$
x=\sum_i \a_i M_iBu_i
$$
with $\a_i\in\dbR$, $M_i\in\cW_{n-1}(A,C)$, and $u_i\in\dbR^m$. Then
$$
Ax=\sum_i \a_i (AM_i)Bu_i,\q  Cx=\sum_i \a_i (CM_i)Bu_i.
$$
Since $AM_i,CM_i\in\cW_n(A,C)$, it follows from \autoref{prop:V(n-1)=Vn} that
$Ax,Cx\in V_n=V_{n-1}$. Thus $AV_{n-1},CV_{n-1}\subseteq V_{n-1}$, 
and the conclusion follows from \autoref{lmm:AC-invariant}.
\end{proof}

We are now ready to state and prove the controllability decomposition 
for system \rf{SDE:state}.

\begin{theorem}[Controllability Decomposition]\label{thm:control-decomposition}
Suppose that $0<k\deq \dim V_{n-1}<n$.  
Then there exists an orthogonal matrix $P\in\dbR^{n\times n}$ such that
\begin{equation}\label{Controllable-Form}
P^\top AP = \begin{pmatrix}A_{11} & A_{12}\\ 0 & A_{22}\end{pmatrix}, \q
P^\top CP = \begin{pmatrix}C_{11} & C_{12}\\ 0 & C_{22}\end{pmatrix}, \q
P^\top B = \begin{pmatrix}B_1 \\ 0\end{pmatrix},	\q
\dim U_{n-1} =k,
\end{equation}
where $A_{11},C_{11}\in\dbR^{k\times k}$, $A_{12},C_{12}\in\dbR^{k\times(n-k)}$,
$A_{22},C_{22}\in\dbR^{(n-k)\times(n-k)}$, $B_1\in\dbR^{k\times m}$, and 
$$
U_{n-1} \deq \Span\big\{\im(MB_1): M\in\cW_{n-1}(A_{11},C_{11})\big\}. 
$$
\end{theorem}

\begin{proof}
Choose an orthogonal basis $\{e_1,\dots,e_k\}$ of $V_{n-1}$ and extend it to an
orthogonal basis $\{e_1,\dots,e_n\}$ of $\dbR^n$. Let
$$
P\deq(e_1,\dots,e_n).
$$
Since $V_{n-1}$ is $[A,C]$-invariant by \autoref{lmm:Vn-invariant}, we have
$$
P^\top AP=\begin{pmatrix}A_{11}&A_{12}\\0&A_{22}\end{pmatrix},\qquad
P^\top CP=\begin{pmatrix}C_{11}&C_{12}\\0&C_{22}\end{pmatrix}
$$
for suitable block matrices of compatible dimensions. Also, since $\im B\subseteq V_{n-1}$,
$$
P^\top B=\begin{pmatrix}B_1\\0\end{pmatrix}
$$
for some $B_1\in\dbR^{k\times m}$. Thus the first three identities 
in \eqref{Controllable-Form} hold. Moreover,
\begin{align*}
k&=\dim V_{n-1} =\dim(P^\top V_{n-1}) =\dim\Span\big\{\im(P^\top MB): M\in\cW_{n-1}(A,C)\big\} \\
&=\dim\Span\big\{\im(P^\top MPP^\top B): M\in\cW_{n-1}(A,C)\big\} \\
&=\dim\Span\big\{\im(MP^\top B): M\in\cW_{n-1}(P^\top AP,P^\top CP)\big\} \\
&=\dim\Span\lt\{\im\!\!\bp MB_1 \\ 0 \ep: M\in\cW_{n-1}(A_{11},C_{11})\rt\} \\
&=\dim\Span\big\{\im MB_1: M\in\cW_{n-1}(A_{11},C_{11})\big\} \\
&=\dim U_{n-1}. 
\end{align*}
This completes the proof.
\end{proof}

We now present the main result of this section, namely, a Hautus-type characterization 
of exact controllability. To this end, we first introduce a linear operator on $\dbS^n$ 
and compute its adjoint.
Endowed with the Frobenius inner product $\lan\cd\,,\cd\ran$, 
$\dbS^n$ is a finite-dimensional inner product space with $\dim\dbS^n={n(n+1)\over2}$. 
For given $A,C\in\dbR^{n\times n}$, define the linear operator $\cL_{\scp(A,C)}:\dbS^n\to\dbS^n$ by
\begin{equation}\label{def:L(M)}
\cL_{\scp(A,C)}(M) \deq MA + A^\top M + C^\top MC, \q M\in\dbS^n.
\end{equation}

\begin{lemma}\label{lmm:adjoint-cL}
The adjoint operator $\cL_{\scp(A,C)}^*$ of $\cL_{\scp(A,C)}$ is given by
$$ 
\cL_{\scp(A,C)}^*(M) = AM + MA^\top + CMC^\top, \q M\in\dbS^n. 
$$
\end{lemma}

\begin{proof}
Write $\cL_{\scp(A,C)}$ and $\cL_{\scp(A,C)}^*$ simply as $\cL$ and $\cL^*$.  
For any $M,N\in\dbS^n$, we have  
\begin{align*}
\lan\cL^*(M),N\ran
&= \lan M,\cL(N)\ran = \tr[M\cL(N)] = \tr[MNA + MA^\top N + MC^\top NC] \\
&= \tr[AMN + MA^\top N + CMC^\top N] \\
&= \lan AM+MA^\top+CMC^\top, N\ran,
\end{align*}
which yields the desired formula for $\cL^*$.
\end{proof}

At first sight, it is not clear whether the operator $\cL_{\scp(A,C)}$ always possesses a nonzero eigenmatrix in $\bar\dbS^n_+$. The next result shows that this is indeed the case.

\begin{proposition}\label{prop:L(P)=rP}
There exist $\l\in\dbR$ and $P\in\bar\dbS^n_+\setminus\{0\}$ such that $\cL_{\scp(A,C)}(P)=\l P$.
\end{proposition}

\begin{proof}
Let $\b\in\dbR$ be the largest eigenvalue of $A+A^\top+C^\top C$. Then, for any $\a>\b/2$,
$$
(A-\a I)+(A-\a I)^\top+C^\top C = A+A^\top+C^\top C-2\a I \les (\b-2\a)I<0.
$$
Hence, by \autoref{lmm:equi-stable}, the system $[A-\a I,C]$ is $L^2$-stable.
Fix such an $\a$. By \autoref{lmm:equi-stable}, the operator
$$
\cK:\dbS^n\to\dbS^n,\q
\cK(\L)\deq \dbE\int_0^\i\varPi_\a(t)^\top \L \varPi_\a(t)dt 
$$
is well defined, maps $\bar\dbS^n_+$ into itself, and satisfies
\begin{align}\label{26-4-7:1}
\cK(\L)(A-\a I)+(A-\a I)^\top\cK(\L)+C^\top\cK(\L)C+\L=0, \q\forall\L\in\dbS^n,
\end{align}
where $\varPi_\a(\cd)$ is the solution of
$$
d\varPi_\a(t)=(A-\a I)\varPi_\a(t)\,dt+C\varPi_\a(t)\,dW(t),\q \varPi_\a(0)=I.
$$
By the Krein--Rutman theorem, there exist $r\ges0$ and $P\in\bar\dbS^n_+\setminus\{0\}$ such that
$$
\cK(P)=r(\cK)P.
$$
Moreover, $r>0$; otherwise,
$$
0=\cK(P)=\dbE\int_0^\infty \varPi_\a(t)^\top P\varPi_\a(t) dt,
$$
which implies $P=0$, a contradiction.
Now substituting $\L=P$ into \rf{26-4-7:1} and using $\cK(P)=rP$, we obtain
\begin{align*}
0 &= r\big[P(A-\a I)+(A-\a I)^\top P+C^\top PC\big]+P
   = r\big(PA+A^\top P+C^\top PC\big)+(1-2\a r)P.
\end{align*}
Hence
$$
\cL_{\scp(A,C)}(P)=PA+A^\top P+C^\top PC = \frac{2\a r-1}{r} P.
$$
This completes the proof.
\end{proof}

The following result provides a Hautus-type characterization of exact controllability.

\begin{theorem}\label{thm:H-test-controllability}
The following statements are equivalent:  
\begin{enumerate}[\rm(i)]
\item System \rf{SDE:state} is exactly controllable.
\item $B^\top H\ne0$ for any eigenmatrix $H\in\bar\dbS^n_+\setminus\{0\}$ of the operator
      $$
      \cL_{\scp(-A,C)}: \dbS^n \to \dbS^n, \q \cL_{\scp(-A,C)}(M) = -MA - A^\top M + C^\top MC. 
      $$ 
\end{enumerate}
\end{theorem}

\begin{proof}
By \autoref{lmm:adjoint-cL}, the adjoint operator of $\cL_{\scp(-A,C)}$ is given by 
$$
\cL^*_{\scp(-A,C)}(M) = -AM -MA^\top + CMC^\top. 
$$
By part (ii) of \autoref{lmm:controllability-Gramian} and \autoref{prop:controllability-kehua}, 
system \rf{SDE:state} is exactly controllable on $[0,T]$ if and only if  
the solution of $P(\cd)\in C([0,T];\dbS^n)$ of
\begin{equation*}\left\{\begin{aligned}
& \dot P(t) + \cL^*_{\scp(-A,C)}(P(t)) + BB^\top = 0,  \q t\in[0,T],\\
& P(T) = 0 
\end{aligned}\right.\end{equation*}
satisfies $P(0)>0$. 

\ms  

Assume first that $P(0)>0$, and let $H\in\bar\dbS^n_+\setminus\{0\}$ be an eigenmatrix of $\cL_{\scp(-A,C)}$
associated with the eigenvalue $\l$. Then 
\begin{align*}
d\lan P(t),H\ran 
&= -\big[\lan \cL^*_{\scp(-A,C)}(P(t)),H\ran + \lan BB^\top,H\ran\big] dt \\
&= -\big[\lan P(t),\cL_{\scp(-A,C)}(H)\ran + \lan BB^\top,H\ran\big] dt \\
&= -\big[\l\lan P(t),H\ran + \lan BB^\top,H\ran\big] dt.
\end{align*}
If $B^\top H=0$, then $\lan BB^\top,H\ran = \tr(BB^\top H) = 0$, and hence
$$
d\lan P(t),H\ran = -\l\lan P(t),H\ran dt, \q \lan P(T),H\ran = 0,
$$
which yields $\lan P(t),H\ran=0$ for all $t\in[0,T]$. 
In particular, $\lan P(0),H\ran=0$, which contradicts the facts that $P(0)>0$ and
$H\in\bar\dbS^n_+\setminus\{0\}$. 

\ms  

Conversely, assume that condition (ii) holds. 
If system \rf{SDE:state} is not exactly controllable, then  $0<k\deq\dim V_{n-1}<n$.
Let $P$ be the orthogonal matrix given by \autoref{thm:control-decomposition}, so that
$$
P^\top AP = \begin{pmatrix}A_{11} & A_{12}\\ 0 & A_{22}\end{pmatrix}, \q
P^\top CP = \begin{pmatrix}C_{11} & C_{12}\\ 0 & C_{22}\end{pmatrix}, \q
P^\top B = \begin{pmatrix}B_1 \\ 0\end{pmatrix}.  
$$
By \autoref{prop:L(P)=rP}, the operator $\cL_{\scp(-A_{22},C_{22})}$ admits an eigenmatrix 
$H_{22}\in\bar\dbS^{n-k}_+\setminus\{0\}$ with eigenvalue $\l$. Define
$$
H\deq P\begin{pmatrix}0&0\\0&H_{22}\end{pmatrix}P^\top
\in \bar\dbS^n_+\setminus\{0\}.
$$
A direct computation using the above block forms shows that 
$$
P^\top \cL_{\scp(-A,C)}(H)P
= \bp0&0\\0&\cL_{\scp(-A_{22},C_{22})}(H_{22})\ep  = \l \bp0&0\\0&H_{22}\ep.
$$
Hence $\cL_{\scp(-A,C)}(H)=\l H$. Moreover, 
$$
B^\top HP = B^\top PP^\top HP = (B_1^\top, 0)\bp 0 & 0 \\ 0 & H_{22}\ep = 0. 
$$
and hence $B^\top H=0$, contradicting (ii). 
Therefore system \rf{SDE:state} is exactly controllable.   
\end{proof}

In analogy with classical Hautus-type tests, one might naturally expect that 
the exact controllability of system \rf{SDE:state} is equivalent to the condition
\begin{equation}\label{BH:strong}
B^\top H \neq 0 \text{ for all eigenmatrices } H \text{ of } \cL_{\scp(-A,C)}.
\end{equation}
Indeed, in the proof of \autoref{thm:H-test-controllability}, any eigenmatrix 
$H$ with $B^\top H=0$ leads to
$$
\frac{d}{dt}\lan P(t),H\ran=-\l\lan P(t),H\ran, \q \lan P(T),H\ran = 0. 
$$
This may suggest that all eigenmatrices should be excluded. 
However, the key contradiction in the proof comes from the positivity property
$$
\lan P(0),H\ran>0 \q\text{whenever } P(0)>0,~H\in\bar\dbS^n_+\setminus\{0\},
$$
which in general fails for indefinite eigenmatrices. 
Thus, condition (ii) of \autoref{thm:H-test-controllability} only requires that
\begin{equation}\label{BH:weak}
B^\top H \neq 0 \text{ for all real positive semidefinite eigenmatrices } H \text{ of }\cL_{\scp(-A,C)}.
\end{equation}
Since \rf{BH:strong} is strictly stronger than \rf{BH:weak}, and the latter is equivalent to 
exact controllability by \autoref{thm:H-test-controllability}, condition \rf{BH:strong} is 
therefore \emph{sufficient but not necessary}. 
We next present a counterexample showing that exact controllability may hold even when 
\rf{BH:strong} fails.

\begin{example}\label{example}
Consider  
$$
A = \begin{pmatrix*}[r] 0 & -1 & 0 \\ 0 & 0 & 0 \\ 1 & 0 & 0 \end{pmatrix*}, \q 
C = \bp 0 &  0 & 0 \\ 1 & 0 & 1 \\  0 & 1 & 0 \ep, \q  
B = \bp 1 \\ 0 \\ 0\ep.
$$
In this example, $n=3$ and 
$$
V_2 = \Span\big\{\im(MB): M\in\cW_2(A,C)\big\} = \im(B,AB,CB,A^2B,C^2B,ACB,CAB)
$$
By a straightforward computation, we have
$$
\rank(B,AB,CB,A^2B,C^2B,ACB,CAB) 
\ges \rank(B,AB,CB) = \rank\!\bp 1 & 0 & 0 \\ 0 & 0 & 1 \\ 0 & 1 & 0 \ep = 3.
$$
Hence $V_2=\dbR^3$. It follows from \autoref{prop:controllability-kehua} that  
the corresponding system \rf{SDE:state} is exactly controllable.  
However, one can verify that the indefinite matrix 
$$
H = \bp 0 & 0 & 0 \\ 0 & 0 & 1 \\ 0 & 1 & 0 \ep \in\dbS^3
$$
is an eigenmatrix of $\cL_{\scp(-A,C)}$ associated with the eigenvalue $\l=1$, while $B^\top H=0$. 
Thus condition \rf{BH:strong} fails, although the system is exactly controllable. 
\end{example}

\section{Stabilizability}  

In this section, we study the $L^2$-stabilizability of system \rf{SDE:state}.
By \autoref{def:stabilizability} and \autoref{lmm:equi-stable}, the notions of
$L^2$-stabilizability, exponential mean-square stabilizability, and asymptotic
mean-square stabilizability are equivalent. Accordingly, in what follows, we simply
speak of stabilizability. 
The main goal of this section is to establish a Hautus-type characterization of
stabilizability; see \autoref{thm:H-test-stability}. We also show that exact
controllability implies stabilizability; see \autoref{prop:control-stable}.

\ms 

We begin by showing that exact controllability implies stabilizability.

\begin{proposition}\label{prop:control-stable}
If system \rf{SDE:state} is exactly controllable, then it is stabilizable. 
\end{proposition}

\begin{proof}
Since the system \rf{SDE:state} is exactly controllable, for every $x\in\dbR^n$, 
there exists a control $(u(\cd),z(\cd))\in L^2_\dbF(0,T;\dbR^m)\times L^2_\dbF(0,T;\dbR^n)$ 
such that $X(T;x,u(\cd),z(\cd))=0$. 
Extend $(u(\cd),z(\cd))$ to $[0,\i)$ by setting
$$
\bu(t) \deq \begin{cases} u(t), & 0 \les t\les T, \\ 0, & t>T,\end{cases} \q 
\bz(t) \deq \begin{cases} z(t), & 0 \les t\les T, \\ 0, & t>T.\end{cases}
$$
Then the corresponding solution satisfies $X(t;x,\bu(\cd),\bz(\cd))=0$ for all $t\ges T$, 
and hence
$$
\dbE\int_0^\i|X(t;x,\bu(\cd),\bz(\cd))|^2dt = \dbE\int_0^T|X(t;x,u(\cd),z(\cd))|^2dt <\i.
$$
Therefore, system \rf{SDE:state} is stabilizable by \autoref{lmm:ACBD}.
\end{proof}

We next establish several preliminary results toward the Hautus-type characterization 
of stabilizability. Recall the operator $\cL_{\scp(A,C)}$ and its adjoint $\cL_{\scp(A,C)}^*$:
$$
\cL_{\scp(A,C)}(M) = MA + A^\top M + C^\top MC, \q 
\cL_{\scp(A,C)}^*(M) = AM + MA^\top + CMC^\top, \q M\in\dbS^n. 
$$
The next lemma gives a stochastic representation of the semigroup generated 
by $\cL_{\scp(A,C)}^*$.

\begin{lemma}\label{lmm:Tt(S)}
Let $(T_t)_{t\ges0}$ be the uniformly continuous semigroup on $\dbS^n$  
generated by $\cL_{\scp(A,C)}^*$, i.e.,
$$
T_t=e^{t\cL_{\scp(A,C)}^*}, \q t\ges0. 
$$
Let $\varPi(\cd)$ be the solution of the matrix SDE \rf{matrix-SDE}. Then 
\begin{equation}\label{eqn:Tt(S)}
T_t(S) = \dbE[\varPi(t)S\varPi(t)^\top], \q\forall S\in\dbS^n. 
\end{equation}
In particular, $T_t$ is positive, that is, 
$$
S \ges 0 \q\Longrightarrow\q T_t(S) \ges 0.
$$
\end{lemma}

\begin{proof}
Write $\cL^*=\cL_{\scp(A,C)}^*$, and let $P(\cd\,;S)$ be the solution of
\begin{equation*}\left\{\begin{aligned}
\dot P(t) &= \cL^*(P(t)) = AP(t) + P(t)A^\top + CP(t)C^\top, \q t\ges0, \\
     P(0) &= S\in\dbS^n.
\end{aligned}\right.\end{equation*}
Then
$$
P(t;S)=T_t(S), \q\forall S\in\dbS^n.
$$
Now let $X(\cd\,;x)$ be the solution of
$$\left\{\begin{aligned}
dX(t) &= AX(t)dt + CX(t)dW(t), \q t\ges0, \\
 X(0) &= x\in\dbR^n.
\end{aligned}\right.$$
Since $X(t;x)=\varPi(t)x$, It\^o's formula gives
$$
\frac{d}{dt}\dbE[X(t;x)X(t;x)^\top] = \cL^*\big(\dbE[X(t;x)X(t;x)^\top]\big),
\q \dbE[X(0;x)X(0;x)^\top]=xx^\top.
$$
By uniqueness of solutions to the above ODE,
\begin{equation}\label{P(t;xx)}
P(t;xx^\top)=\dbE[\varPi(t)xx^\top\varPi(t)^\top], \q\forall x\in\dbR^n.
\end{equation}
Finally, every $S\in\dbS^n$ can be written as 
$$
S=\sum_{k=1}^n x_kx_k^\top-\sum_{k=1}^n y_ky_k^\top
$$
for some $x_k,y_k\in\dbR^n$. By linearity and \rf{P(t;xx)},
\begin{align*}
P(t;S)
&=\sum_{k=1}^n P(t;x_kx_k^\top)-\sum_{k=1}^n P(t;y_ky_k^\top) =\dbE[\varPi(t)S\varPi(t)^\top].
\end{align*}
Hence \eqref{eqn:Tt(S)} holds. The positivity of $T_t$ is then immediate.
\end{proof}

Recall from \autoref{lmm:equi-stable} that the stability of system $[A,C]$ is equivalent to
$$
\exists P>0\quad\text{such that}\quad \cL_{\scp(A,C)}(P)<0.
$$
The following proposition provides a useful counterpart for the opposite-sign condition
$$
\exists P>0\quad\text{such that}\quad \cL_{\scp(A,C)}(P)>0,
$$
and will play a key role in the proof of the stabilizability criterion.

\begin{proposition}\label{prop:P>0+LP>0} 
The following statements are equivalent:
\begin{enumerate}[\rm(i)] 
\item There exists $P>0$ such that $\cL_{\scp(A,C)}(P)>0$.  
\item If $S\in\bar\dbS^n_+$ and $\cL^*_{\scp(A,C)}(S)\les0$, then $S=0$.
\item If $H\in\bar\dbS^n_+\setminus\{0\}$ and $\cL^*_{\scp(A,C)}(H)=\l H$
      for some $\l\in\dbR$, then $\l>0$.    
\end{enumerate} 
\end{proposition}

\begin{proof}
For simplicity, we write $\cL=\cL_{\scp(A,C)}$ and $\cL^*=\cL_{\scp(A,C)}^*$. 

\ms  
 
(i) $\Ra$ (ii): Suppose that $\cL(P)>0$ for some $P>0$. 
If there existed $S\in\bar\dbS^n_+\setminus\{0\}$ such that $\cL^*(S)\les0$, then
$$
0<\lan \cL(P),S\ran=\lan P,\cL^*(S)\ran\les0,
$$
a contradiction.

\ms   

(ii) $\Ra$ (i): Suppose, to the contrary, that there exists no $P>0$ such that
$\cL(P)>0$. Then the convex sets
$$
\dbS^n_+=\{X\in\dbS^n:X>0\} \q\text{and}\q  \cC\deq \cL(\dbS^n_+)
$$
are disjoint. Since $\dbS^n$ is finite-dimensional and $\dbS^n_+$ is open, the separation
theorem (see \cite[page 11, Corollary 1.2.10]{Cheridito2013}) yields a nonzero $S\in\dbS^n$ 
and $\a\in\dbR$ such that
\begin{align}\label{<SX><SY>}
\lan S,X\ran \les \a, \q\forall X\in\dbS^n_+,  \qq\text{and}\qq
\lan S,Y\ran \ges \a, \q\forall Y\in\cC.  
\end{align}
Fix $X_0\in\dbS^n_+$. Since $tX_0\in\dbS^n_+$ for all $t>0$, \rf{<SX><SY>} implies
$$
t\lan S,X_0\ran=\lan S,tX_0\ran\les \a, \q\forall t>0.
$$
Letting $t\to0$ yields $\a\ges0$, while letting $t\to\i$ shows that
$\lan S,X_0\ran\les0$. Since $X_0\in\dbS^n_+$ is arbitrary, we obtain
$$
\lan S,X\ran\les0,\q\forall X\in\dbS^n_+,
$$
and hence $S\les0$. On the other hand, \rf{<SX><SY>}  and $\a\ges0$ imply
$$
\lan \cL^*(S),X\ran=\lan S,\cL(X)\ran\ges0,\q\forall X\in\dbS^n_+.
$$
Thus $\cL^*(S)\ges0$. It follows that 
$$
-S\ges0, \q \cL^*(-S)\les0, \q -S\ne0,
$$
which contradicts (ii).

\ms 

(ii) $\Ra$ (iii): This is immediate.

\ms 

(iii) $\Ra$ (ii): We argue by contradiction. 
Suppose that there exists $S\in\bar\dbS^n_+\setminus\{0\}$ such that $\cL^*(S)\les0$. 

\ms 

{\it Step 1}. Let $(T_t)_{t\ges0}$ be the semigroup generated by $\cL^*$. 
Since $\cL^*(S)\les0$, \autoref{lmm:Tt(S)} implies   
\begin{equation}\label{T(L(S))<0}
T_t(\cL^*(S)) \les0, \q\forall t\ges0.
\end{equation}
Because $\cL^*$ and $T_t$ commute, we have 
$$
{d\over dt}T_t(S) = \cL^*(T_t(S)) = T_t(\cL^*(S)), \q T_0(S) = S.
$$
Consequently, by integrating the above identity and using \rf{T(L(S))<0}, we obtain
\begin{equation}\label{T(S)<S}
T_t(S) = S + \int_0^t T_s\big(\cL^*(S)\big) ds \les S, \q\forall t\ges0. 
\end{equation}

{\it Step 2}. Define the following sets:   
\begin{align*}
\cI_{\scS} &\deq \big\{X\in\bar\dbS^n_+ : \exists\,\a\ges0 \text{ such that } X \les \a S\big\}
              = \big\{X\in\bar\dbS^n_+ : \im X \subseteq \im S\big\}, \\
V_{\scS} &\deq \big\{X\in\dbS^n : \im X \subseteq \im S\big\}.
\end{align*}
Clearly, $\cI_{\scS}$ is a cone contained in the linear subspace $V_{\scS}$. Moreover, 
\begin{equation}\label{V=cI-cI}
V_{\scS} = \cI_{\scS} - \cI_{\scS}.
\end{equation}
By \rf{T(S)<S} and the positivity of $T_t$, for any $X\in\cI_{\scS}$, i.e., 
$0\les X\les\a S$ for some $\a\ges0$, we have 
$$
0 \les T_t(X) \les \a T_t(S) \les \a S.
$$
Consequently, combining this with \rf{V=cI-cI}, we deduce the invariance property
\begin{equation}\label{eq:V-invariance}
T_t(\cI_{\scS})\subseteq \cI_{\scS}, \q T_t(V_{\scS})\subseteq V_{\scS}, \q\forall t\ges0. 
\end{equation}
Now equip $V_{\scS}$ with the order-unit norm induced by $S$: 
$$
\|X\|_{\scS} \deq \inf\{\a\ges0 : -\a S \les X \les \a S\}.
$$
Let $\wt T_t$ be the restriction of $T_t$ to $V_{\scS}$.  
By the positivity of $T_t$ and \rf{T(S)<S}, we have
$$
-\a S \les X \les \a S \q\Longrightarrow\q -\a S \les T_t(X) \les \a S; \qq\forall X\in V_{\scS}.
$$
Taking the infimum over admissible $\a$ yields
$$
\|\wt T_t(X)\|_{\scS} \les \|X\|_{\scS}, \q\forall\, X\in V_{\scS}.
$$
Consequently, the operator norm of $\wt T_t$ on $(V_{\scS},\|\cd\|_{\scS})$ satisfies
$\|\wt T_t\|_{\scp V_{\scalebox{0.4}{$S$}} \to V_{\scalebox{0.4}{$S$}}} \les 1$,
and hence the spectral radius of $\wt T_t$ satisfies 
\begin{equation}\label{r(wt-Tt)<1}
r(\wt T_t) \les 1.
\end{equation}

{\it Step 3}. Consider the nonempty, compact, convex subset of $(V_{\scS},\|\cd\|_{\scS})$:
$$
K \deq \{X\in\cI_{\scS} : \tr(X)=1 \}.
$$
The following holds:
\begin{equation}\label{tr(TX)>0}
\tr[T_t(X)]>0, \q\forall X\in K.
\end{equation}
Indeed, $X\in K$ implies that $X\ges0$ and $\tr(X)=1$, hence $X\ne0$.
Since $T_t$ is positive, we have $T_t(X)\ges0$, and hence $\tr[T_t(X)]\ges0$. 
If $\tr[T_t(X)]=0$, then $T_t(X)=0$. Since $T_t$ is invertible, this yields $X=0$, a contradiction.  
Properties \rf{tr(TX)>0} and \rf{eq:V-invariance} allow us to define, for each $t\ges0$, 
a continuous mapping $f_t:K\to K$ by 
$$
f_t(X) \deq \frac{T_t(X)}{\tr[T_t(X)]}, \q X\in K. 
$$
By Brouwer's fixed-point theorem, 
there exists $H_t\in K$ such that $f_t(H_t)=H_t$, that is, 
$$ 
T_t(H_t)=\tr[T_t(H_t)]H_t.
$$
Set $\b_t\deq\tr[T_t(H_t)]>0$. Since $H_t\in K\subseteq V_{\scS}$, 
the above identity shows that $(\b_t, H_t)$ is an eigenpair of the restriction
$\wt T_t$ of $T_t$ to $V_{\scS}$. Consequently, by \rf{r(wt-Tt)<1}, 
$$
0<\b_t\les r(\wt T_t) \les 1.
$$
 
{\it Step 4}. 
From Step 3, we see that for each $t\ges0$, $T_t$ admits an eigenpair $(\b_t,H_t)$ satisfying 
\begin{equation}\label{eq:H-eigen} 
0<\b_t\les 1, \q T_t(H_t)=\b_t H_t, \q H_t\in\bar\dbS^n_+\setminus\{0\}. 
\end{equation}
Let $\si(\cL^*)=\{\l_1,\dots,\l_m\}$ be the spectrum of $\cL^*$. By the spectral mapping theorem, we have  
$$
\si(T_t) = \si\big(e^{t\cL^*}\big) = \big\{e^{t\l_1},\dots,e^{t\l_m}\big\}.  
$$
Choose $t>0$ such that 
\begin{equation}\label{eq:noncollision}
e^{t\l_i} \neq e^{t\l_j}, \q\forall i\neq j.
\end{equation}
Then each eigenvalue of $T_t$ has the form $e^{t\mu}$ for a unique $\mu\in\si(\cL^*)$. Moreover,
\begin{equation}\label{eq:kernel}
\ker(e^{t\cL^*}-e^{t\mu}I) = \ker(\cL^*-\mu I).
\end{equation}
From \rf{eq:H-eigen} we have $e^{t\cL^*}H_t=\b_tH_t$.  
Writing $\b_t=e^{t\mu}$, where $\mu$ is uniquely determined by \rf{eq:noncollision}, 
and using the kernel identity \rf{eq:kernel}, we obtain
\begin{equation}\label{eq:Lstar-eigen}
\cL^*(H_t)=\mu H_t.
\end{equation}
Since $0<\b_t\les1$, it follows that $\mu=\tfrac{1}{t}\ln\b_t\les0$. Thus,
$$
H_t\in\bar\dbS^n_+\setminus\{0\}, \q \cL^*(H_t)=\mu H_t, \q \mu\les0,
$$
which contradicts condition (iii). Therefore, (ii) holds.  
\end{proof}

We next consider the special case $B=0$, which serves as a basic building block 
for the general stabilizability analysis of system \rf{SDE:state}.

\begin{proposition}\label{prop:stabilizable:B=0}
The following statements are equivalent:
\begin{enumerate}[\rm(i)]
\item The system 
      \begin{equation}\label{SDE:B=0}\left\{\begin{aligned}
      dX(t) &= [AX(t)+Cz(t)]dt + z(t)dW(t),  \q t\ges 0, \\
       X(0) &= x 
      \end{aligned}\right.\end{equation}
      is stabilizable. 
\item There exists $P\in\dbS^n_+$ such that $\cL_{\scp(-A,C)}^*(P)>0$.
\item If $S\in\bar\dbS^n_+$ and $\cL_{\scp(-A,C)}(S)\les0$, then $S=0$. 
\item If $H\in\bar\dbS^n_+\setminus\{0\}$ and $\cL_{\scp(-A,C)}(H)=\l H$
      for some $\l\in\dbR$, then $\l>0$. 
\end{enumerate}
\end{proposition}

\begin{proof}
It suffices to prove that (i) and (ii) are equivalent, since the remaining equivalences 
follow directly from \autoref{prop:P>0+LP>0}. 
By definition and \autoref{lmm:equi-stable}, system \rf{SDE:B=0} is stabilizable 
if and only if there exist $\Si\in\dbS^n_+$ and $\Th\in\dbR^{n\times n}$ such that 
\begin{equation}\label{Si:SDE-B=0}
\Si A+ A^\top\Si + \Si C\Th +\Th^\top C^\top\Si + \Th^\top\Si\Th
=\Si(A+C\Th) + (A+C\Th)^\top\Si + \Th^\top\Si\Th <0. 
\end{equation}

(i) $\Ra$ (ii): Suppose that \rf{Si:SDE-B=0} holds for some 
$(\Si,\Th)\in\dbS^n_+\times\dbR^{n\times n}$. Using
$$
\Si C\Th + \Th^\top C^\top\Si + \Th^\top\Si\Th 
= (\Th^\top + \Si C\Si^{-1})\Si(\Th + \Si^{-1}C^\top\Si) - \Si C\Si^{-1}C^\top\Si,
$$
we obtain from \rf{Si:SDE-B=0} that 
$$
\Si A + A^\top\Si - \Si C\Si^{-1}C^\top\Si 
< -(\Th^\top + \Si C\Si^{-1})\Si(\Th + \Si^{-1}C^\top\Si) \les 0.
$$
Pre- and post-multiplying the above by $P\deq\Si^{-1}$, we get 
\begin{equation}\label{25-10-22:1}
-\cL_{\scp(-A,C)}^*(P) = AP + P A^\top -CPC^\top  <0.
\end{equation}

(ii) $\Ra$ (i): Conversely, suppose that \rf{25-10-22:1} holds for some $P>0$.
Let $\Si \deq P^{-1}$ and $\Th \deq -\Si^{-1}C^\top\Si$. Then 
$$
\Si A + A^\top\Si - \Si C\Si^{-1}C^\top\Si <0, \q 
\Si C\Th +\Th^\top C^\top\Si + \Th^\top\Si\Th = -\Si C\Si^{-1}C^\top\Si.
$$
It follows that \rf{Si:SDE-B=0} holds.   
\end{proof}

To establish the Hautus-type characterization of stabilizability for system \rf{SDE:state}, 
we next prove the final reduction step: under the controllability decomposition, 
the stabilizability of system \rf{SDE:state} is equivalent to that of the uncontrollable 
subsystem \rf{eq:subsystem}.

\begin{proposition}\label{prop:stable-decomposition}
Suppose that $0<k\deq \dim V_{n-1}<n$, where $V_{n-1}$ is defined by \rf{def:Vk}. 
Let $P\in\dbR^{n\times n}$ be the orthogonal matrix satisfying \rf{Controllable-Form}. 
Then system \rf{SDE:state} is stabilizable if and only if the subsystem
\begin{equation}\label{eq:subsystem}\left\{\begin{aligned}
dY_2(t) &= [A_{22}Y_2(t) + C_{22}v_2(t)]dt + v_2(t)dW(t), \\
 Y_2(0) &= y_2\in\dbR^{n-k}
\end{aligned}\right.\end{equation}
is stabilizable.   
\end{proposition}

\begin{proof}
Let
\begin{equation}\label{eqn:Y=PTX}
Y(t)=\bp Y_1(t)\\ Y_2(t)\ep \deq P^\top X(t),\q v(t)=\bp v_1(t)\\ v_2(t)\ep \deq P^\top z(t),
\end{equation}
where $Y_1(t),v_1(t)\in\dbR^k$ and $Y_2(t),v_2(t)\in\dbR^{n-k}$. Then, by \rf{Controllable-Form},
\begin{equation}\label{eqn:dY=}\left\{\begin{aligned}
dY_1(t) &=\big[A_{11}Y_1(t)+A_{12}Y_2(t)+B_1u(t)+C_{11}v_1(t)+C_{12}v_2(t)\big]dt + v_1(t)dW(t), \\
dY_2(t) &=\big[A_{22}Y_2(t)+C_{22}v_2(t)\big]dt + v_2(t)dW(t).
\end{aligned}\right.\end{equation}
Clearly, the stabilizability of systems \rf{SDE:state} and \rf{eqn:dY=} is equivalent.  
Indeed, $\bigl(\!\begin{smallmatrix}F\\K\end{smallmatrix}\!\bigr)$ is a stabilizer of 
system \rf{SDE:state} if and only if $\bigl(\!\begin{smallmatrix}FP\\P^\top\!KP\end{smallmatrix}\!\bigr)$ 
is a stabilizer of system \rf{eqn:dY=}.   

\ms 

If system \rf{eqn:dY=} is stabilizable, then its subsystem \rf{eq:subsystem} is obviously stabilizable.
Conversely, suppose that subsystem \rf{eq:subsystem} is stabilizable. Then there exists 
$K_2\in\dbR^{(n-k)\times(n-k)}$ such that the closed-loop system
$$
dY_2(t) = (A_{22} + C_{22}K_2)Y_2(t) dt + K_2Y_2(t) dW(t)
$$
is stable. Equivalently, there exists $P_{22}\in\dbS^{n-k}_+$ such that
$$
\D_{22}\deq P_{22}(A_{22} + C_{22}K_2)+(A_{22} + C_{22}K_2)^\top P_{22} + K_2^\top P_{22}K_2<0. 
$$
Moreover, by \autoref{thm:control-decomposition}, the system   
\begin{equation}\label{eq:A11B1C11}\left\{\begin{aligned}
dY_1(t) &= [A_{11}Y_1(t) + B_1u(t) + C_{11}v_1(t)]dt + v_1(t)dW(t), \\ 
 Y_1(0) &= y_1\in \dbR^k
\end{aligned}\right.\end{equation}
is exactly controllable and therefore stabilizable by \autoref{prop:control-stable}.  
Thus there exist $F_1\in\dbR^{m\times k}$, $K_1\in\dbR^{k\times k}$, and $P_{11}\in\dbS^{k}_+$ 
such that 
$$
\D_{11}\deq P_{11}(A_{11}+B_1F_1+C_{11}K_1)+(A_{11}+B_1F_1+C_{11}K_1)^\top P_{11}+K_1^\top P_{11}K_1<0.
$$
Now set  
\begin{align*}
\wh A= \bp \wh A_{11} & \wh A_{12} \\ 0 & \wh A_{22}\ep 
  \deq \bp A_{11} + B_1F_1 + C_{11}K_1 &  A_{12} + C_{12}K_2 \\ 0 &  A_{22} + C_{22}K_2\ep, \q
K \deq \bp K_1 & 0 \\ 0 & K_2\ep. 
\end{align*}
We claim that there exists $P\in\dbS^n_+$ such that 
$$
P\wh A +\wh A^{\,\top}P + K^\top PK <0, 
$$
and hence $\biggl(\!\begin{smallmatrix}F_1 & 0 \\ K_1 & 0 \\ 0 & K_2 \end{smallmatrix}\!\biggr)$ 
stabilizes system \rf{eqn:dY=}.   
To prove this, fix any $\a>0$ and consider the block-diagonal matrix
$$
P(\a)\deq\diag(\a P_{11},P_{22})\in\dbS^n_+.
$$
A straightforward computation yields 
$$
P(\a)\wh A + \wh A^{\,\top} P(\a) + K^\top P(\a)K
=\bp \a\D_{11} & \a P_{11}\wh A_{12} \\ \a\wh A_{12}^{\,\top} P_{11} & \D_{22}\ep. 
$$
Since $\D_{22}<0$, the Schur complement shows that this matrix is negative definite if and only if
$$
\a\(\D_{11}-\a P_{11}\wh A_{12}\D_{22}^{-1}\wh A_{12}^{\,\top} P_{11}\)=
\a\D_{11}-(\a P_{11}\wh A_{12})\D_{22}^{-1}(\a\wh A_{12}^{\,\top} P_{11})<0.
$$
Because $\D_{11}<0$, the above inequality holds for all sufficiently small $\a>0$. 
\end{proof}

We now present a Hautus-type characterization of stabilizability for 
the backward-structured stochastic system \rf{SDE:state}.

\begin{theorem}\label{thm:H-test-stability}
The following statements are equivalent:   
\begin{enumerate}[\rm(i)]
\item System \rf{SDE:state} is stabilizable.
\item For every $\l\les0$ and every $H\in\bar\dbS^n_+$, 
      $$
      \cL_{\scp(-A,C)}(H) = \l H, ~ B^\top H=0 \q\Rightarrow\q H=0. 
      $$ 
\end{enumerate}
\end{theorem}

\begin{proof} 
(i) $\Ra$ (ii): Let $V_{n-1}$ be defined by \rf{def:Vk}, and set $k\deq\dim V_{n-1}$.
If $k=n$, then system \rf{SDE:state} is exactly controllable,
and (ii) follows from \autoref{thm:H-test-controllability}. 
If $k=0$, then $B=0$, and (ii) follows from \autoref{prop:stabilizable:B=0}(iv). 

\ms   
  
Now assume $0<k<n$, and let $P$ be the orthogonal matrix in \autoref{thm:control-decomposition} 
satisfying \rf{Controllable-Form}. 
Then system \rf{eq:A11B1C11} is exactly controllable, and system \rf{eq:subsystem} is stabilizable. 
Take $\l\les0$ and $H\in\bar\dbS^n_+$ such that 
$$
\cL_{\scp(-A,C)}(H) = -HA - A^\top H + C^\top HC = \l H, \q B^\top H=0.
$$
Write $H$ in block form: 
$$
H =\bp H_{11} & K \\ K^\top & H_{22}\ep, \q 
H_{11}\in\bar\dbS^k_+, \q 
H_{22}\in\bar\dbS^{n-k}_+, \q 
K\in\dbR^{k\times (n-k)},
$$
and define
$$
\wt H \deq  P^\top HP=\bp\wt H_{11} & \wt K \\ \wt K^\top & \wt H_{22}\ep, \q 
\wt H_{11}\in\bar\dbS^k_+, \q 
\wt H_{22}\in\bar\dbS^{n-k}_+, \q 
\wt K\in\dbR^{k\times (n-k)}.
$$
Since $P$ is orthogonal, $B^\top H=0$ implies 
\begin{align}\label{block:B}
0 &= B^\top HP = B^\top PP^\top HP = (B_1^\top,0)\wt H = (B_1^\top\wt H_{11}, B_1^\top\wt K).
\end{align}
Moreover, using $\cL_{\scp(-A,C)}(H)=\l H$ and \rf{Controllable-Form}, we obtain
\begin{align}\label{block:H}
\l\wt H &= P^\top(-HA - A^\top H + C^\top HC)P = \bp\L_{11} & \L_{12}\\ \L_{12}^\top & \L_{22}\ep, 
\end{align}
where
\begin{align*}
\L_{11} &= -\wt H_{11}A_{11} - A_{11}^\top\wt H_{11} + C_{11}^\top\wt H_{11}C_{11}, \\ 
\L_{12} &= -\big(\wt H_{11}A_{12} + \wt K A_{22}+A_{11}^\top\wt K\big) 
           +\big(C_{11}^\top\wt H_{11}C_{12} + C_{11}^\top\wt KC_{22}\big), \\
\L_{22} &= -\big(\wt K^\top A_{12} + A_{12}^\top\wt K + \wt H_{22}A_{22} + A_{22}^\top\wt H_{22}\big) \\
 &\hp{=\ } +\!\big(C_{12}^\top\wt H_{11}C_{12} + C_{12}^\top\wt K C_{22} + C_{22}^\top\wt K^\top C_{12} 
           +C_{22}^\top\wt H_{22}C_{22}\big).
\end{align*}
Hence, comparing blocks in \rf{block:H} and \rf{block:B} gives 
$$
\L_{22}=\l\wt H_{22}, \q 
\cL_{\scp(-A_{11},C_{11})}(\wt H_{11}) = \l\wt H_{11}, \q 
B_1^\top\wt H_{11}=0. 
$$
Since system \rf{eq:A11B1C11} is exactly controllable, 
\autoref{thm:H-test-controllability} implies that $\wt H_{11}=0$. 
As $\wt H\ges0$, this forces $\wt K=0$. Consequently,
$$
\l\wt H_{22} = \L_{22} = -\wt H_{22}A_{22} - A_{22}^\top\wt H_{22} + C_{22}^\top\wt H_{22}C_{22}
=\cL_{\scp(-A_{22},C_{22})}(\wt H_{22}). 
$$ 
Since system \rf{eq:subsystem} is stabilizable and $\l\les0$, 
\autoref{prop:stabilizable:B=0} yields $\wt H_{22}=0$. 
Hence, $\wt H=0$, and thus $H=0$.  

\ms 

(ii) $\Ra$ (i): Suppose, to the contrary, that system \rf{SDE:state} is not stabilizable.
Since exact controllability implies stabilizability, we must have $\dim V_{n-1}<n$.
If $\dim V_{n-1}=0$, then $B=0$. By \autoref{prop:stabilizable:B=0}(iv), 
there exists $H\in\bar\dbS_+^n$, $H\ne0$, and $\l\les0$ such that 
$\cL_{\scp(-A,C)}(H)=\l H$, which contradicts (ii).  

\ms

Now assume $0<k<n$. By \autoref{prop:stable-decomposition}, 
the subsystem \rf{eq:subsystem} is not stabilizable.   
Hence, by \autoref{prop:stabilizable:B=0}(iv), there exists $H_{22}\in\bar\dbS_+^{n-k}$, 
$H_{22}\ne0$, and $\l\les0$ such that 
$$
\cL_{\scp(-A_{22},C_{22})}(H_{22})=\l H_{22}.
$$
Define $H\deq P\,\diag(0,H_{22})P^\top\in\bar\dbS_+^n\setminus\{0\}$, 
where $P\in\dbR^{n\times n}$ is the orthogonal matrix in \autoref{prop:stable-decomposition}.
Then a direct computation using \rf{Controllable-Form} shows that
\begin{align*}
P^\top\cL_{\scp(-A,C)}(H)P 
&= \diag(0,\cL_{\scp(-A_{22},C_{22})}(H_{22})) = \l\,\diag(0,H_{22}).
\end{align*}
Consequently, 
$$
\cL_{\scp(-A,C)}(H)=\l H, \q H\in\bar\dbS^n_+\setminus\{0\},  \q \l\les0. 
$$
However,  
$$
B^\top H = (B_1^\top,0)\,\diag(0,H_{22})P^\top =0, \q H\ne0. 
$$
This again contradicts (ii). Therefore system \rf{SDE:state} is stabilizable. 
\end{proof}

\end{document}